\documentclass[12pt,reqno]{amsart}

\usepackage[
  margin=1in,
  includefoot,
  footskip=30pt,
]{geometry}
\usepackage[nodisplayskipstretch]{setspace}

\usepackage{amsfonts}
\usepackage{mathrsfs}
\usepackage{amssymb}

\usepackage{amsthm}

\usepackage{enumitem}

\usepackage{tikz-cd}

\usepackage{hyperref}
\usepackage{cleveref}

\theoremstyle{definition}
\newtheorem{definition}{Definition}[section]
\newtheorem{example}[definition]{Example}

\theoremstyle{plain}
\newtheorem{theorem}[definition]{Theorem}
\newtheorem*{theorem*}{Theorem}
\newtheorem{maintheorem}{Theorem}

\newtheorem{proposition}[definition]{Proposition}
\newtheorem{lemma}[definition]{Lemma}
\newtheorem*{question}{Question}
\newtheorem{corollary}[definition]{Corollary}

\def\op{\operatorname}
\def\A{\mathbf{A}}
\def\C{\mathbb{C}}
\def\D{\mathfrak{D}}
\def\E{\mathscr{E}}
\def\F{\mathscr{F}}
\def\G{\mathscr{G}}
\def\H{\mathscr{H}}
\def\I{\mathscr{I}}

\def\K{\mathscr{K}}
\def\L{\mathscr{L}}
\def\M{\mathcal{M}}
\def\O{\mathscr{O}}
\def\P{\mathbf{P}}

\def\Spec{\op{Spec}}
\def\Fitt{\mathscr{F}itt}
\def\Hom{\op{Hom}}
\def\Ext{\op{Ext}}

\def\sExt{\mathscr{E}\kern -.5pt xt}
\newcommand*{\sHom}{\mathscr{H}\kern -.5pt om}

\renewcommand{\tilde}{\widetilde}

\def\coker{\op{coker}}

\def\im{\op{im}}
\def\rank{\op{rank}}

\def\depth{\op{depth}}

\def\proj.dim{\op{proj.dim}}

\def\a{\underline{a}}
\def\b{\underline{b}}
\def\c{\underline{c}}
\def\d{\underline{d}}

\begin{document}

\title{Biliaison of sheaves}
\author{Mengyuan Zhang}
\address{Department of Mathematics, University of California, 
Berkeley, CA 94720}
\email{myzhang@berkeley.edu}

\begin{abstract}
We define an equivalence relation among coherent sheaves on a projective variety called biliaison.
We prove the existence of sheaves that are minimal in a biliaison class in a suitable sense, and show that all sheaves in the same class can be obtained from a minimal one using certain deformations and other basic moves.
Our results generalize the main theorems of liaison theory of subvarieties to sheaves, and provide a framework to study sheaves and subvarieties simultaneously.
\end{abstract}

\maketitle 

\section*{Introduction}
Liaison theory originated from the work of M. Noether in 1882 that classified algebraic curves in $\P^3_\C$ of degree $\le 20$.
It has since become instrumental in the study of the Hilbert scheme of projective spaces \cite{MDP90,Bolondi94}.
Over the last decades, tremendous progress and generalizations were made thanks to Peskine, Szpiro, Rao, Lazarsfeld, Ballico, Bolondi, Migliore, Nollet, Strano, Hartshorne and many others.
Briefly, a (geometric) link is a pair of subschemes that are residual to each other in a complete intersection.
An even linkage class consists of those subschemes that can be obtained from one another using even numbers of links.
The following is a summary of the main results of liaison theory.
\begin{enumerate}[leftmargin=*]
\item The degrees of subschemes in an even linkage class are bounded below, and those with the minimal degree differ by a deformation preserving cohomologies \cite{Migliore83,BM89}.
\item Any subscheme can be obtained from a minimal one in its even linkage class using certain elementary moves \cite{LR83,BBM91,Nollet96,Strano04,Hartshorne03}. 
\item Even linkage classes are in bijection with stable equivalence classes of primitive bundles \cite{Rao78,Rao81}.
\end{enumerate}
In particular, the numerical invariants of subschemes in an even linkage class can be systematically deduced from those of a minimal one.
We refer to the book \cite{Migliore98} for an introduction to liaison theory, and to \cite[Introduction]{Hartshorne03} for a survey of these results.

\medskip

In this article we study an equivalence relation among sheaves on a projective variety called biliaison, which generalizes even linkages of varieties.
We prove the following.
\begin{enumerate}[leftmargin=*, label=(\alph*)]
\item There are minimal sheaves in each biliaison class under a suitable preorder (\Cref{Existence}). 
Those that are minimal can be obtained from each other using a rational deformation preserving cohomologies (\Cref{Rigid}).
\item Any sheaf can be obtained from a minimal one in its biliaison class using rigid deformations and certain basic moves (\Cref{LazarsfeldRao}).
\end{enumerate}
Hartshorne \cite{Hartshorne03} has proven that (c) biliaison classes of sheaves are in bijection with stable equivalence classes of primitive sheaves. 
These results (a)-(c) give us satisfactory extensions of (1)-(3) above.
The biliaison theory for the special case of rank two reflexive sheaves on $X = \P^3_k$ was established by Buraggina in \cite{Buraggina99} using Serre correspondences and results from liaison theory of curves. 
Our method provides a substantially simplified treatment as well as stronger theorems even in this special case.

\medskip

Aside from structure theorems for biliaison classes, we also prove a sufficient criterion for a sheaf to be minimal (\Cref{Sufficient}), generalizing the criterion of \cite{LR83} for curves to be minimal in $\P^3_k$. 
As application, we prove that the Horrocks-Mumford bundle on $\P^4_\C$ is a minimal reflexive sheaf.
We also provide a conceptual proof that indecomposable rank two bundles on $\P^3_k$ are minimal reflexive sheaves, which is the main result in \cite{Buraggina99} obtained from the computations in \cite{MDP92}.

\medskip

Let $X$ be a projective variety with a very ample line bundle $\O(1)$.
We say a sheaf $\F$ is a descendant of $\E$, and $\E$ is an ancestor of $\F$, if there is an exact sequence of the form 
\[
0 \to \bigoplus_{i = 1}^u \O(-a_i)\to \E\oplus \bigoplus_{i = 1}^v \O(-b_i) \to \F \to 0
\]
for some $a_i,b_i \in\mathbb{Z}$. 
We say $\F_1\sim \F_2$ iff they have a common ancestor, and define the equivalence relation generated by the relation $\sim$ to be biliaison.
It is a theorem of Rao that two codimension two subschemes $Y,Z$ are evenly linked iff $\I_Y$ and $\I_Z(\delta)$ are biliaison equivalent for some $\delta\in\mathbb{Z}$.
Thus biliaison of sheaves generalizes even linkage of subschemes.

\bigskip

Biliaison equivalence is closely related to stable equivalence.
If $H^1(\O(l)) = 0$ for all $l$, then $\sim$ is itself an equivalence (thus the same as biliaison) and is also called ``psi-equivalence" in \cite{Hartshorne03}.
If two sheaves $\F_1$ and $\F_2$ are stably equivalent, i.e. $\F_1\oplus \O(\a)\cong \F_2\oplus \O(\b)$ for some sequences $\a,\b$, then clearly $\F_1\sim \F_2$. 
The converse is only true when both $\F_i$ are primitive, i.e. $\Ext^1(\F_i,\O(l)) = 0$ for all $l$.  
It follows that biliaison classes that contain a primitive sheaf are in bijection with stable equivalence classes of primitive sheaves.
For example, any biliaison class containing a torsion-free sheaf also contains a primitive sheaf.

\bigskip

Biliaison equivalence has the notion of Serre correspondence built in.
A Serre correspondence is an extension of the form $0 \to \O(a) \to \E \to \I_Y \to 0$,
where $\E$ is a rank two bundle and $\I_Y$ is an ideal sheaf.
This correspondence connects the study of rank two bundles with that of codimension two local complete intersections. 
Hartshorne \cite{Hartshorne80} generalized this correspondence by relaxing the condition on $\E$ to be reflexive, and studied rank two reflexive sheaves on $\P^3_k$ via curves in $\P^3_k$ and vice versa.
If $\E$ and $\E'$ are two rank two reflexive sheaves corresponding to the same curve $C$, then $\E$ and $\E'$ are clearly closely related. 
It turns out that we can generalize ``elementary biliaisons", the basic moves needed to construct subschemes from a minimal one, by considering two-sided Serre correspondences. 

\bigskip

Our work paves the way for a program to describe the moduli $\M$  (in a broad sense) of (e.g. semistable, stable) torsion-free sheaves on a smooth projective variety $X$ with irregularity zero.
First, we partition $\M$ by biliaison equivalence into pieces $\M_{\E}$ that are labeled by stable equivalence classes of primitive sheaves. 
We need to wisely choose a very ample line bundle and classify the stable equivalence classes of primitive sheaves on $X$.
Next, we partition each $\M_{\E}$ into $\M_{\E,\Sigma}$ by $\Sigma$ functions.
If we can determine the $\Sigma$ functions of the minimal sheaves in $\M_{\E}$, then we can systematically deduce the numerical invariants of sheaves in $\M_{\E,\Sigma}$, the dimensions of the pieces $\M_{\E,\Sigma}$ and the dimensions of tangent spaces etc.
If we restrict to rank one torsion free sheaves $X = \P^3_k$, which are exactly ideal sheaves up to twist, then this program has been successfully carried out in \cite{MDP90}, where the appropriate moduli $\M$ is the Hilbert scheme.
Another example is the description of the moduli of bundles on $X = \P^n_k$ in the class of the zero sheaf  \cite{Bundles}.
The theory of biliasion of sheaves studies (the ideal sheaves of) subschemes, reflexive sheaves, bundles in one fell swoop, while utilizing the relations between pieces of the moduli that correspond to sheaves of different ranks.

\section{Notations and assumptions}

Throughout, we work on a Cohen-Macaulay projective variety $X$ over an infinite field $k$ with a fixed very ample line bundle $\O(1)$.
All sheaves in consideration are coherent on $X$.
We write $H^i_*(\F)$ to denote $\bigoplus_{l\in\mathbb{Z}} H^i(\F(l))$.
We use underlined letters such as $\a$ to denote a finite sequence of integers $(a_i)_{i = 1}^u$, and write $\O(\a)$ in place of $\bigoplus_{i = 1}^u \O(-a_i)$ for brevity.

\bigskip

Recall that a sheaf $\E$ satisfies \emph{Serre's condition} $(S_m)$ if $\depth \E_x \ge \min(m,\dim \O_x)$ for all $x\in X$.
In this article, we need a slightly stronger condition. 
We say that $\E$ satisfies $(S_m^+)$ if $\E$ satisfies $(S_m)$ and is locally-free in codimension $m$.
The latter condition follows from the former if $\E_x$ has finite projective dimension over $\O_x$ at all points $x$ of codimension $m$ in $X$. 
For example, the extra condition $(^+)$ makes no difference when $X$ is regular in codimension $m$.
All results in this article remain valid if one replaces the condition $(S_m^+)$ with suitable conditions of sufficient depth and locally-freeness in certain codimensions, or other conditions that behave well in a short exact sequence (e.g. \cite[condition $(T)$]{Hartshorne03}).

\bigskip

For readers who are interested in vector bundles on projective varieties and who wish to get a gist of the ideas in this article without too much commutative algebra, it is advisable to take $m = \dim X$ throughout and replace all occurrences of ``$(S_m^+)$ sheaves'' with ``bundles''.

\section{Lattice structure}

\begin{definition}
An \emph{$m$-reduction} of a sheaf $\E$ is an injective map of the form $\varphi:\O(\a) \to \E$ where $\coker \varphi$ satisfies $(S_m^+)$. 
We define the \emph{shape} of the reduction $\varphi$ to be the sequence $\a$ sorted in \textbf{ascending} order.
\end{definition}

 In order for a sheaf $\E$ to admit an $m$-reduction, it must satisfy $(S_m^+)$ itself. 
A main result in \cite{Basic} states that if $\E$ has rank $r\ge m$ and satisfies $(S_m^+)$, then it admits an $m$-reduction whose cokernel has rank $m$.

Other than the case of bundles, the following $(S_m^+)$ sheaves are of interest to us. 
If $X$ is regular in codimension one, then $\E$ satisfies $(S_1^+)$ iff it is torsion-free.
In this case, a rank one sheaf $\E$ satisfies $(S_1^+)$ iff it is isomorphic to $\I\otimes \L$, where $\I$ is an ideal sheaf and $\L$ is a line bundle. 
If $X$ is regular in codimension two, then $\E$ satisfies $(S_2^+)$ iff $\E$ is reflexive, i.e. the natural map $\E \to \E^{**}$ is an isomorphism.

The following examples of $m$-reductions in the literature motivated our definition.

\begin{example}
Suppose  $\dim X = m$ and $\E$ is locally-free.
An $m$-reduction of $\E$ is an injective map $\varphi: \O(\a)\to \E$ whose cokernel is locally-free, which corresponds to sections $s_1,\dots, s_u$ of $\E$ in degrees $\a$ that are linearly independent at every fiber over $x\in X$.
Serre proved that if $\rank \E = r >m$, then there is an $m$-reduction of $\E$ of rank $r -m$ \cite[p148]{Mumford66}.
\end{example}

\begin{example}\label{GH}
Let $X$ be a smooth surface.
A sheaf $\E$ satisfies $(S_2^+)$ iff it is locally-free. 
Griffith and Harris \cite{GH78} proved that points $Z$ on $X$ arises from an extension 
\[
0 \to \L \to \E \to \I_Z \to 0 
\]
for some line bundle $\L$ and rank two locally-free sheaf $\E$ if and only if $Z$ is a local complete intersection and satisfies the Cayley-Bacharach property with respect to the linear system $|\L\otimes \omega_X|$.
This result sets up a correspondence between rank two locally-free sheaves with points on a smooth surface.
\end{example}

\begin{example}\label{Serre}
Let $X = \P^3_k$. 
Serre \cite{Serre58} proved that a curve $C$ in $\P^3_k$ arises as the vanishing scheme of a rank two locally-free sheaf $\E$, i.e. there is an extension of the form 
\[
0\to \O(-a) \to \E \to \I_C \to 0
\]
 if and only if $C$ is a local complete intersection and $\omega_C \cong \O_C(a-4)$.
This result sets up a correspondence between rank two locally-free sheaves on $\P^3_k$ and subcanonical local complete intersection curves.
\end{example}

\begin{example}\label{Reflexive}
Hartshorne \cite{Hartshorne80} generalized the above to a correspondence between rank two reflexive sheaves and generic complete intersection curves in $\P^3_k$. 
A Serre correspondence is defined to be an extension of the form
\[
0 \to \O(-a) \to \E \to \I_C \to 0
\]
where $\E$ is a rank two reflexive sheaf and $\I_C$ is an ideal sheaf.  
\end{example}

If we consider the shapes of all $m$-reductions $\phi:\O(\a)\to \E$ of the same sheaf $\E$, we see that they are partially ordered in a natural way.
We define this partial order more generally on the set of all finite non-decreasing sequences of integers.

\begin{definition}
Let $\a$ and $\b$ be two finite non-decreasing sequences of integers.
\begin{enumerate}[leftmargin=*]
\item For an integer $l$, we define $\Sigma(\a,l)$ to be the number of entries of $\a$ that is $\le l$.

\noindent Note that the non-decreasing function $\Sigma(\a,-): \mathbb{Z} \to \mathbb{N}$ determines the non-decreasing sequence $\a$.
\item We write $\a \le \b$ if $\Sigma(\a,l) \le \Sigma(\b,l)$ for all $l\in\mathbb{Z}$. 
\item Let $\a \vee \b$ be the non-decreasing sequence $\c$ determined by the property that 
\[
\Sigma(\c,l) = \min(\Sigma(\a,l),\Sigma(\b,l)),\quad \forall l\in\mathbb{Z}.
\]
\item Let $\a\wedge \b$ be the non-decreasing sequence $\c$ determined by the property that  
\[
\Sigma(\c,l) = \max(\Sigma(\a,l),\Sigma(\b,l)),\quad \forall l\in\mathbb{Z}.
\]
\end{enumerate}

Let $\mathfrak{S}$ denote the set of finite non-decreasing sequences of integers.

It is easy to see that the poset $(\mathfrak{S},\le)$ is a lattice with meet $\vee$ and join $\wedge$.
\end{definition}

\begin{example}\label{ExLattice}
If $\a = (1,3,4)$ and $\b = (2,2)$, then $\a\wedge \b = (1,2,4)$ and $ \a\vee \b = (2,3)$.
\end{example}

Here is an equivalent way to determine $\a\vee \b$ and $\a \wedge \b$ without writing down $\Sigma(\a,-)$ and $\Sigma(\b,-)$. 
We illustrate on the previous example.
First append $\infty$ to the shorter sequence till the lengths match up: $\a = (1,3,4)$ and $\b = (2,2,\infty)$. 
Then $\a\vee \b$ and $\a \wedge \b$ are given by the position-wise maximum and minimum, with $\infty$ interpreted as a non-entry.

\bigskip

Our first main result is that the shapes of the $m$-reductions of a given sheaf $\E$ form a subsemilattice of $\mathfrak{S}$.

\begin{maintheorem}[Semilattice theorem]\label{Main1}
For a fixed $m\ge 1$, the shapes of $m$-reductions of a given sheaf $\E$ is a subsemilattice of $\mathfrak{S}$.
When $m = 1$, the shapes of $1$-reductions of a given sheaf $\E$ is a sublattice of $\mathfrak{S}$. 
\end{maintheorem}

A semilattice is a poset with meet. 
A subsemilattice is a subposet inheriting the same meet from the ambient semilattice, analogously for a sublattice. 
The conclusions of \Cref{Main1} follow immediately from \Cref{Join} and \Cref{Meet}, whose proofs will occupy the remainder of the section.

\bigskip

The next lemma is a generalization of \cite[Lemma 3.6]{Nollet96}, where the quotients are required to have rank one and satisfy $(S_1^+)$.  
Essentially we prove that $m$-reductions are open among the affine variety of morphisms.

\begin{lemma}\label{Replace}
Let $\phi,\psi$  be $m$-reduction of $\E$ with shapes $(a_i)_{i = 1}^u$ and $(b_i)_{i = 1}^v$.
Denote by $J \subseteq\{1,\dots, \min(u,v)\}$ the subset of indices where $a_j = b_j$ for all $j \in J$.
There are $m$-reductions $\phi',\psi'$ of $\E$ with shapes $(a_i)_{i = 1}^u$ and $(b_i)_{i = 1}^v$, such that
if $s_1',\dots, s_u'$ and $t_1',\dots, t_v'$ are twisted sections of $\E$ corresponding to $\phi'$ and $\psi'$ respectively, then $s_j' = t_j'$ for all $j\in J$.
\end{lemma}
\begin{proof}
Let $s_1,\dots, s_u$ be sections of $\E$ in degrees $a_1,\dots, a_u$ corresponding to $\phi$. 
Let $\epsilon$ be an arbitrary index. 
We claim that for a general choice of sections $s_\epsilon'$ of $\E$ in degree $a_\epsilon$, the map $\phi':\O(\a)\to \E$ given by $s_1,\dots,s_\epsilon', \dots ,s_u$ is an $m$-reduction of $\E$.
The conclusion of the lemma then follows by replacing the sections $s_j$ and $t_j$ by a common general section of $\E$ of degree $a_j = b_j$ for every $j \in J$.
 
Let $V$ be the finite dimensional $k$-vector space $H^0(\E(a_\epsilon))$, and let $\A = \Spec \op{Sym} V^*$ be the affine space parametrizing these sections. 
Consider the scheme $X' = X \times_k \A$ and the pullback sheaves $\E'$ of $\E$ as well as $\O(\a)'$ of $\O(\a)$ from $X$. 
There is a map $\Phi: \O(\a)' \to \E'$ defined by the following property. 
Suppose $p$ is a $k$-point of $\A$ corresponding to the sections $s_\epsilon'$ in $V$, then the fiber $\Phi_p : \O(\a) \to \E$ at $p$ is given by the sections $s_1,\dots, s_\epsilon',\dots, s_u$.
Let $Z$ denote the subscheme of $X'$ cut out by the Fitting ideal $\Fitt_{r-u}(\coker \Phi)$. 
The subscheme $Z$ contains points in $X'$ where $\coker \Phi$ cannot be generated by $\le r-u$ elements locally.
Since $X'$ is proper over $\A$, semicontinuity of fiber dimensions over the target implies that the points $p\in A$ such that $Z_p$ has dimension $< \dim X -m$ are open in $A$. 
Since $X$ is a projective variety over a field, dimension and codimension are complementary by Noether normalization. 
It follows that the set of points $p\in \A$ such that $\op{codim}(Z_p,X \otimes_k k(p)) > m$ form an open set $U$. 
If $p\in \A$ is such a point, then $\coker \Phi_p$ has rank $r$ on $X\otimes_k k(p)$ and therefore we have an exact sequence
\[
0 \to \bigoplus_{i = 1}^u \O(-a_i) \otimes_k k(p) \to \E \otimes_k k(p) \to \coker \Phi_p \to 0.
\]
Since $\coker \Phi_p$ is locally-free in codimension $m$, it satisfies $(S_m^+)$ by an application of the depth lemma to the above sequence. 
Finally, the existence of $\phi$ gives us a $k$-point of $U$.
Since $k$ is infinite, it follows that the $k$-points of $U$ are dense in $V$. 
\end{proof}

To simplify the language in the next proof, we recall the definition of basic elements.

\begin{definition}
A subsheaf $\E'$ of $\E$ is \emph{$w$-basic} at $x\in X$ if 
\[
\dim_{k(x)} (\E/\E')\otimes k(x) \le \dim_{k(x)} \E \otimes k(x)-w.
\] 
Let $s_1,\dots, s_u$ be twisted sections of $\E$ corresponding to a map $\phi:\O(\a) \to \E$. 
We write $(s_1,\dots, s_u)$ for the subsheaf $\im \phi$ and say that $s_1,\dots, s_u$ are \emph{basic in $\E$ at $x$} if $(s_1,\dots, s_u)$ is $u$-basic in $\E$ at $x$. 
\end{definition}

By Nakayama's lemma, a subsheaf $\E'$ is $w$-basic in $\E$ at $x$ iff $\E'$ contains at least $w$ of a system of minimal generators of $\E$ at $x$. 
If $\E$ is locally free at $x \in X$ and $s_1,\dots, s_u$ are basic in $\E$ at $x$, then $\E/(s_1,\dots, s_u)$ is also locally-free at $x$.

\begin{theorem}\label{Join}
For a fixed $m\ge 1$, if there are $m$-reductions $\phi,\psi$ of $\E$ with shapes $(a_i)_{i = 1}^u$ and $(b_i)_{i = 1}^v$, then there is an $m$-reduction $\xi$ of $\E$ with shape $(a_i)_{i = 1}^u\wedge (b_i)_{i = 1}^v$.
\end{theorem}
\begin{proof}
By induction, we may reduce to proving the following statement.
Set $\epsilon$ to be the largest integer in the interval $[0,\min(u,v)]$ such that $a_i = b_i$ for all $1\le i\le \epsilon$.
If $\epsilon = \min(u,v)$, then $(a_i)_{i = 1}^u$ is a subsequence of $(b_i)_{i = 1}^v$ or vice versa, and their join is just the longer sequence among the two.
The statement of the theorem is true in this case.
Assume without loss of generality that $\epsilon < \min(u,v)$ and $b_{\epsilon+1} > a_{\epsilon+1}$. 
We claim that there exists an $m$-reduction of $\E$ with shape $a_1,\dots, a_{\epsilon+1}, b_{\epsilon+2},\dots, b_v$.

 \textbf{Step 1:} Using \Cref{Replace}, we may assume $\phi$ and $\psi$ are given by twisted sections $s_1,\dots, s_u$ and $t_1,\dots, t_v$ respectively, such that $s_j = t_j$ for all $1\le j\le \epsilon$.

\textbf{Step 2:} 
Let $Y$ be any subvariety of $X$ of codimension $<m$, we claim that $(\coker \psi)_Y$ is torsion-free on $Y$. 
Since $\coker \psi$ satisfies $(S_m^+)$, we see that $(\coker \psi)_Y$ is locally-free in codimension one on $Y$. 
By Krull's principal ideal theorem, if a module $M$ has a zerodivisor $r$, then any minimal prime $P$ above $(r)$ has height one.
In particular, the image of $r$ in the localization would remain a zerodivisor on $M_P$. 
Since $(\coker \psi)_Y$ is locally-free in codimension one on $Y$, we conclude that $(\coker \psi)_Y$ must be torsion-free on $Y$.

\textbf{Step 3:} 
We claim that $t_1,\dots, t_v, s_{\epsilon+1}$ are basic in $\E$ at all points of codimension $\le m-1$. 
Suppose not, let $y \in X$ be a point of codimension $\le m-1$ and let $Y = \overline{\{y\}}$ be the corresponding subvariety. 
If $t_1,\dots, t_v, s_{\epsilon+1}$ are not basic in $\E$ at $y$, then the image of the corresponding map $\psi_Y':\bigoplus_{i = 1}^v \O_Y(-b_i) \oplus \O_Y(-a_{\epsilon+1}) \to \E_Y$ has rank $v$ on $Y$, the same rank as $\im \psi_Y$. 
We obtain the following commutative diagram of exact sequences
\[
\begin{tikzcd}
0 \arrow[r] & \bigoplus_{i = 1}^v \O_Y(-b_i) \arrow[d,"\alpha"]\arrow[r,"\psi_Y"] & \E_Y\arrow{r} \arrow[d,equal] & (\coker \psi)_Y \arrow{r} \arrow[d,"\beta"] & 0\\
0 \arrow{r} & \im \psi_Y' \arrow{r} & \E_Y \arrow{r} & \coker \psi_Y' \arrow{r}  & 0.
\end{tikzcd}
\]
The upper complex is exact because it is exact at the generic point $y$ of $Y$ and $\bigoplus_{i = 1}^v \O_Y(-b_i)$ is torsion-free.
The snake lemma implies that $\coker \alpha \cong \ker \beta$, which vanishes at the generic point $y$ of $Y$ since $\rank \im \psi_Y' = v$, therefore $\ker \beta$ is torsion on $Y$. 
Since $(\coker \psi)_Y$ is torsion-free by the above step, we conclude that $\ker \beta= 0$. 
This means that $\psi_Y'$ factors through $\psi_Y$. 
However, any map from $\O_Y(-a_{\epsilon+1})$ to $\O_Y(-b_i)$ is zero for $i > \epsilon$ since $X$ is integral and $\O(1)$ is very ample.
It follows that $\O_Y(-a_{\epsilon+1}) \to \E_Y$ factors through $\bigoplus_{i = 1}^\epsilon \O_Y(-b_i) \to \E_Y$. 
This means that $t_1,\dots, t_\epsilon, s_{\epsilon+1}$ are not basic at $y$, which is a contradiction since $t_j = s_j$ for all $1\le j\le \epsilon$ and $s_1,\dots, s_{\epsilon+1}$ are basic at $y$.

\textbf{Step 4:} 
Let $Z$ be the subscheme defined by the Fitting ideal $\Fitt_{r-1}(\E/(t_1,\dots, t_v, s_{\epsilon+1}))$, where $r = \rank \coker \psi$. 
The subscheme $Z$ contains all points in $X$ where $\E/(t_1,\dots, t_v, s_{\epsilon+1})$ cannot be generated by $\le r-1$ elements.
Since $\E$ is locally-free in codimension $m$, we see that $Z$ contains no point in $X$ of codimension $\le m-1$ by the previous step.
Therefore $Z$ contains at most finitely many points in $X$ of codimension $m$. 
Let $B$ denote this finite set of points of codimension $m$ in $X$ where $t_1,\dots, t_v, s_{\epsilon+1}$ are not basic in $\E$.

The idea is to fix the basicness of $s_1,\dots, s_{\epsilon+1}, t_{\epsilon+2},\dots, t_v$ at one point in $B$ at a time by modifying a section $t_i$ to $t_i+r_i t_{\epsilon+1}$ for some suitable $r_i\in H^0(\O(b_i-b_{\epsilon+1}))$, without worsening the basicness at the remaining points in $B$. 

At each point $x\in B$, if $s_1,\dots, s_{\epsilon+1}, t_{\epsilon+2},\dots, t_v$ are basic in $\E$ then we do nothing.
If not, we can find $t_i$ for $\epsilon+2\le i\le v$ such that $s_1,\dots, s_{\epsilon+1}, t_{\epsilon+2},\dots, t_i$ have the same basicness in $\E$ at $x$ as $s_1,\dots, s_{\epsilon+1}, t_{\epsilon+2},\dots, t_{i-1}$. 
Since $\O(1)$ is very ample, there exists a form $r_i\in H^0(\O(b_i-b_{\epsilon+1}))$ that does not vanish at $x$. 
Let $\lambda\in k$ be an undetermined nonzero scalar, then $s_1,\dots, s_{\epsilon+1}, t_{\epsilon+2},\dots,t_i',\dots, t_v$ are basic in $\E$ at $x$, where $t_i' = t_i+\lambda r_i t_{\epsilon+1}$.  
By \cite[Lemma 1.2]{Basic}, for all but finitely many choices of $\lambda$ the sections $s_1,\dots, s_{\epsilon+1}, t_{\epsilon+2},\dots,t_i',\dots, t_v$ maintain the same amount of basicness as $s_1,\dots, s_{\epsilon+1}, t_{\epsilon+2},\dots, t_v$ at the remaining points in $B$. 
We choose such a nonzero $\lambda$ and go to the next point in $B$ with the modified sections $s_1,\dots, s_{\epsilon+1}, t_{\epsilon+2},\dots,t_i',\dots, t_v$ as input, and carry out the same procedure. 
Eventually, we arrive at sections $s_1,\dots, s_{\epsilon+1}, t'_{\epsilon+2},\dots, t_v'$ that are basic in $\E$ at all points in $B$, where $t_i' = t_i+r_i t_{\epsilon+1}$ for some $r_i\in H^0(\O(b_i-b_{\epsilon+1}))$. 
The sections $s_1,\dots, s_{\epsilon+1}, t'_{\epsilon+2},\dots, t_v'$ are basic in $\E$ at all points of codimension $\le m$ outside of $B$ since $t_1,\dots, t_v, s_{\epsilon+1}$ are basic in $\E$ at these points. 
It follows that the map $\xi: \bigoplus_{i = 1}^{\epsilon+1} \O(-a_i) \oplus \bigoplus_{i = \epsilon+2}^v \O(-b_i) \to \E$ corresponding to $s_1,\dots, s_{\epsilon+1}, t'_{\epsilon+2},\dots, t_v'$ is an $m$-reduction of $\E$.
\end{proof}

\Cref{Join} is a generalization of \cite[Lemma 2.1]{BBM91}.
The above proof for $(S_m^+)$ sheaves is more subtle.
At its core, \Cref{Join} is about the codimension of certain minors of matrix extensions.
The affine versions of these theorems are well-known results of basic element theory by Eisenbud-Evans \cite{EE72} and others.
The extension of these results to the projective case were proven in \cite{Basic}.
Note that this procedure gives us a way to make new bundles from old ones.

\bigskip

The next theorem is similar in spirit with \Cref{Join}, but the proof requires a slightly different approach. 
We include the proof here for the sake of rigor and completeness.
\begin{theorem}\label{Meet}
If there are $1$-reductions $\phi,\psi$ of $\E$ with shapes $(a_i)_{i = 1}^u$ and $(b_i)_{i = 1}^v$, then there is a $1$-reduction $\xi$ of $\E$ with shape $(a_i)_{i = 1}^u\vee(b_i)_{i = 1}^v$.
\end{theorem}
\begin{proof}
Without loss of generality, assume that $u \le v$. 
By the remark below \Cref{ExLattice}, we see that $(a_i)_{1 = 1}^u \vee (b_i)_{i = 1}^v = (a_i)_{1 = 1}^u \vee (b_i)_{i = 1}^u$.
Certainly $\psi': \bigoplus_{i = 1}^v \O(-a_i)\to \E$ is an $m$-reduction of $\E$ if $\psi$ is. 
We thus reduce to the case where $u =v$. 

Let $D(\phi,\psi)$ denote the number of indicies where the shapes of $\phi$ and $\psi$ differ.
We prove the assertion by induction on $D(\phi,\psi)$.
When $D(\phi,\psi) = 0$ there is nothing to prove. 
Suppose $D(\phi,\psi)>0$. 
Let $s_1,\dots,s_u$ and $t_1,\dots, t_u$ be sections of $\E$ corresponding to $\phi$ and $\psi$ respectively. 
Let $J\subseteq \{1,\dots, u\}$ be the subset of indicies $j$ where $a_j = b_j$. 
By \Cref{Replace}, we may assume that $s_j = t_j$ for all $j\in J$.
We claim that there is an index $\epsilon \in \{1,\dots, u\} - J$ where $s_1,\dots, s_u, t_\epsilon$ are basic in $\E$ at the generic point $\eta$ of $X$. 
Suppose not, then every map $ \O(\a)\oplus \O(-b_i) \to \E$ factors through $\phi: \O(\a) \to \E$ by the same argument in step 3 of \Cref{Meet}. 
This would mean that $\psi:\O(\b) \to \E$ factors through $\phi:\O(\a) \to \E$. 
But the factor map $\O(\b) \to \O(\a)$ must drop rank along the determinant hypersurface since $\a\ne \b$, and thus so must $\psi$. 
This is a contradiction to the fact that $\psi$ does not drop rank in codimension one. 
Thus we find $\epsilon$ such that $s_1,\dots, s_u, t_\epsilon$ are basic in $\E$ at the generic point. 
We may assume without loss of generality that $b_\epsilon > a_\epsilon$, otherwise we reverse the role of $\phi$ and $\psi$. 
The same argument in step 4 of \Cref{Join} shows that $s_1,\dots, s_u, t_\epsilon$ are basic in $\E$ at all but finitely many codimension one points in $X$, and thus so are $s_1,\dots, \hat{s_\epsilon},\dots, s_u, t_\epsilon$. 
Carrying out the same procedure in step 4 of \Cref{Join}, we can find a suitable $r\in H^0(\O(b_\epsilon-a_\epsilon))$ such that $s_1,\dots, \hat{s_\epsilon},\dots, s_u, t_\epsilon'$ are basic at all points in $X$ of codimension $\le 1$, where $t_\epsilon' = t_\epsilon+r\cdot s_\epsilon$. 
The corresponding map $\xi:\O(-b_\epsilon)\oplus \bigoplus_{i\ne \epsilon} \O(-a_i) \to \E$ is thus a $1$-reduction. 
Since $D(\phi,\xi) < D(\phi,\psi)$, by induction we find a $1$-reduction $\eta$ of $\E$ with shape $\a \vee \c = \a\vee\b$, where $\c$ is the shape of $\xi$.
\end{proof}

\begin{example}
Continuing \Cref{ExLattice}.
Suppose there are extensions 
\[
0 \to \O(-1)\oplus \O(-3)\oplus \O(-4) \to \E \to \F \to 0
\]
\[
0 \to \O(-2)\oplus \O(-2) \to \E \to \F' \to 0
\]
where $\F,\F'$ are locally-free.
By \Cref{Join}, there is an extension of the form
\[
0 \to \O(-1)\oplus \O(-2) \oplus \O(-4) \to \E \to \F'' \to 0
\]
where $\F''$ is locally-free. 
By \Cref{Meet}, there is an extension of the form
\[
0 \to \O(-2)\oplus \O(-3) \to \E \to \F''' \to 0
\]
where $\F'''$ is $(S_1^+)$.
We do not know if we can always make $\F'''$ locally-free.
\end{example}

\section{Biliaison of sheaves}
In this section, we define the biliaison equivalence of sheaves.
We prove a weak version of the structure theorem for a biliaison class, which says that $(S_m^+)$ sheaves in a biliaison class can be obtained from one another using certain deformations and other basic moves.

\begin{definition}\label{Descendant}

If there is an extension of the form 
\[
0 \to \O(\a) \to \E \oplus \O(\b) \to \F \to 0,
\]
then we say $\F$ is a \emph{descendant} of $\E$ and $\E$ is an \emph{ancestor} of $\F$.
Let $\D(\E)$ denote the collection of all descendants of $\E$, and let $\D_m(\E)$ denote the collection of all descendants of $\E$ that satisfy $(S_m^+)$.  
Two sheaves $\F$ and $\G$ are \emph{related}, written $\F\sim \G$, if they share a common ancestor.
The equivalence relation among sheaves on $X$ generated by $\sim$ is called \emph{biliaison equivalence}.
\end{definition}

Note that if $\F\in \D_m(\E)$, then $\E\oplus \O(\a)$ satisfies $(S_m^+)$ for some $\a$. 
By the characterization of depth using vanishing of local cohomologies, we have $\depth \E_x = \depth (\E\oplus \O(\a))_x$ for any $x\in X$.
We see that $\E$ satisfies $(S_m^+)$ iff $\E\oplus \O(\a)$ satisfies $(S_m^+)$ for some $\a$.
Therefore a sufficient and necessary condition for $\D_m(\E)$ to be nonempty is that $\E$ satisfies $(S_m^+)$.

\bigskip

\Cref{Descendant} is motivated by the following well-known theorem of Rao \cite{Rao81}, strengthened in \cite{Nollet96} and \cite{Hartshorne03}.

\begin{theorem*}[Rao-Nollet-Hartshorne]
Suppose $H^1_*(\O_X) = 0$. 
Two  pure codimension two subschemes $Y,Z$ of $X$ are evenly linked iff $\I_Y, \I_Z(\delta) \in \D_1(\F)$ for a sheaf $\F$ and an integer $\delta$.
If $X$ is Gorenstein in codimension two, then $\F$ can be chosen to be reflexive. 
If $X$ is regular, then $Y$ is Cohen-Macaulay iff $Z$ is Cohen-Macaulay iff $\F$ can be chosen to be locally-free.
\end{theorem*}

We briefly recall what it means for two subschemes to be evenly linked for the unfamiliar readers.
We refer to \cite{Migliore98} for a complete treatment of liaison theory.

A \emph{(geometric) link} is a pair $(Y,Z)$ of pure codimension $r$ subschemes of $X$ that are residual to each other in a complete intersection. 
Two subschemes are \emph{evenly linked} if they can be obtained from one another using even number of links.

Equivalently, a link is a pair of (generalized) divisors $(Y,Z)$ on a codimension $(r-1)$ complete intersection $K$ such that $Y$ is linearly equivalent to $nH-Z$, where $H$ is the hyperplane class of $K$ and $n$ is an integer.
One may similarly define an \emph{elementary biliaison}: a pair of generalized divisors $(Y,Z)$ on a codimension $(r-1)$ complete intersection $K$ such that $Y$ is linearly equivalent to $Z+nH$ for some integer $n$.
We refer to \cite{Hartshorne94} for the theory of generalized divisors on a Gorenstein scheme to make sense of this definition.

\bigskip

We now generalize the notions of Serre correspondence and elemantary biliaison to sheaves.


\begin{definition}\label{BasicMoves}
An \emph{$(S_m^+)$-Serre correspondence} is an $m$-reduction of a sheaf $\E$ of the form $\varphi:\O(-a) \to \E$ for some integer $a$.
An \emph{elementary $(S_m^+)$-biliaison from $\F$ to $\G$} is a pair of $(S_m^+)$-Serre correspondences $\phi:\O(-a) \to \F$ and $\psi:\O(-b) \to \G$
where $\coker \phi \cong \coker \psi$.
The \emph{height} of the elementary biliaison $(\phi,\psi)$ is the integer $a-b$. 
The elementary $(S_m^+)$-biliaison is \emph{increasing} if the height is positive, and \emph{decreasing} otherwise.

Let $\E$ be a sheaf on $X$, let $T$ be a rational variety and let $p:X\times_k T \to X$ be the natural projection. 
If for some $\a$ there is a map $\Phi_T: p^*\O(\a) \to p^*\E$ that is fiber-wise injective over $T$, then we call $\coker \Phi_T$ a \emph{rigid family} of sheaves on $X$.
\end{definition}

If there are extensions $0 \to \O(\a) \to \E \to \F \to 0$ and $0\to \O(\a) \to \E \to \G \to 0$ for the same $\a$ and sheaf $\E$, then $\F$ and $\G$ belong in a rigid family by the proof of \Cref{Replace}.
The converse is true trivially.
In particular, if there is an elementary biliaison of height zero from $\F$ to $\G$, then $\F$ and $\G$ belong in a rigid family. 

\begin{example}
There is an elementary $(S_2^+)$-biliaison between two rank two reflexive sheaves $\E$ and $\E'$ in $\P^3_k$ iff they correspond to the same curve $C$ in $\P^3_k$ in the sense of \cite{Hartshorne80}.
\end{example}

We briefly explain how elementary $(S_m^+)$-biliaisons of sheaves generalize elementary biliaisons of codimension two subvarieties. 
Let $\I_Y$ and $\I_Z$ be ideal sheaves of pure codimension two subschemes $Y$ and $Z$ of $X$. 
If there is an elementary $(S_1^+)$-biliaison between $\I_Y$ and $\I_Z(\delta)$, then we have exact sequences 
\[
0 \to \O(-a) \to \I_Y \to \I_{Y/K} \to 0\]
\[
0 \to \O(-a+\delta) \to \I_Z(\delta) \to \I_{Z/K}(\delta) \to 0,
\]
where $K$ is a hypersurface in the linear system $|\O(a)|$ containing $Y,Z$ such that $\I_{Y/K}\cong \I_{Z/K}(\delta)$ \cite[Prop 3.5]{Hartshorne03}.
Conversely, if there is a hypersurface $K$ in $|\O(a)|$ containing $Y,Z$ such that $\I_{Y/K}\cong \I_{Z/K}(\delta)$, the maps $\O(-a) \to \I_Y$ and $\O(-a+\delta) \to \I_Z(\delta)$ form an elementary $(S_1^+)$-biliaison of height $\delta$ from $\I_Y$ to $\I_Z(\delta)$. 

\bigskip

The main result of this section is that $(S_m^+)$ sheaves in a biliaison class can be obtained from one another using finitely many elementary biliaisons, rigid deformations and at most one $m$-reduction.

\begin{maintheorem}[Weak structure theorem]\label{Weak} For $m\ge 1$, if $\F$ and $\G$ are $(S_m^+)$ sheaves in the same biliaison class, then there are $(S_m^+)$ sheaves $\F = \F_0,\dots, \F_l = \G$, such that $\F_i$ and $\F_{i+1}$ are related in one of the following ways:
\begin{enumerate}
\item[(a)] there is an elementary $(S_{m-1}^+)$-biliaison from $\F_i$ to $\F_{i+1}$, 
\item[(b)] $\F_i$ and $\F_{i+1}$ belong in a rigid family, 
\item[(c)] there is an $m$-reduction $\phi:\O(\a) \to \F_{i+1}$ with $\coker \phi\cong \F_i$ or vice versa.
\end{enumerate}
We need (c) at most once. 
If $\rank \F =\rank \G$, then we do not need (c). 
\end{maintheorem}
\begin{proof}
Since biliaison equivalence is generated by the relation $\sim$, it is enough to prove the assertion when $\F$ and $\G$ have the same ancestor.
By definition, there are extensions 
\[0\to \O(\a) \xrightarrow{\phi} \E\oplus \O(\b) \to \F \to 0\]
\[
0 \to \O(\a') \xrightarrow{\psi} \E \oplus \O(\b') \to \G \to 0.
\]
We may consider two $m$-reductions of the same sheaf $\E' := \E\oplus \O(\b) \oplus \O(\b')$ given by 
\[
\phi \oplus \op{Id}: \O(\a)\oplus\O(\b') \to \E'\]
\[\psi \oplus \op{Id}: \O(\a')\oplus \O(\b) \to \E'.
\] 
In doing so, we reduce to the following: the cokernels of two $m$-reductions $\phi:\O(\a) \to \E$ and $\psi:\O(\b) \to \E$ of the same sheaf $\E$ are related by finitely many steps in the manner $(a)-(c)$.

Let $\a = (a_i)_{i = 1}^u$ and $\b = (b_i)_{i = 1}^v$, and let $s_1,\dots, s_u$ and $t_1,\dots, t_v$ be twisted sections of $\E$ corresponding to $\phi$ and $\psi$ respectively. 
We proceed by induction on $D(\phi,\psi)$, the number of indices where $(a_i)_{i =1}^u$ and $(b_i)_{i = 1}^v$ differ, including those $i$ where only one of $a_i,b_i$ is defined.
If $D(\phi,\psi) = 0$, then $\coker \phi$ and $\coker \psi$ are related in manner (b) by the proof of \Cref{Replace}. 

Suppose $D(\phi,\psi)>0$. 
Let $\epsilon$ be the largest integer in $[0,\min(u,v)]$ where $a_i= b_i$ for all $1\le i\le \epsilon$.
By \Cref{Replace} again, we may replace $\phi$ and $\psi$ in manner (b) and assume that $s_i = t_i$ for $1\le i\le \epsilon$.
If $\epsilon = \min(u,v)$, then $(a_i)_{i = 1}^u$ is a subsequence of $(b_i)_{i = 1}^v$ (or vice versa), and we see that $\coker \phi$ and $\coker \psi$ are related in manner (c). 

We now discuss the case where $\epsilon<\min(u,v)$.
Interchanging $\phi$ and $\psi$ if necessary, we may assume $a_{\epsilon+1}<b_{\epsilon+1}$. 
By the proof of \Cref{Meet}, there is an $m$-reduction $\xi:\O(\underline{c}) \to \E$ corresponding to twisted sections $s_1,\dots, s_{\epsilon+1}, t_{\epsilon+2}',\dots, t_v'$, where $t_i' = t_i + r_i\cdot t_{\epsilon+1}$ for some $r_i\in H^0(\O(-b_i-b_{\epsilon+1}))$.
Since $D(\phi,\xi)$ is smaller than $D(\phi,\psi)$, by the induction hypothesis the sheaves $\coker \phi$ and $\coker \xi$ are related by finitely many steps of manner $(a) - (c)$.
 
To finish the proof, we show that $\coker \psi$ and $\coker \xi$ are related in manner (a). 
In step (3) of the proof of \Cref{Meet}, we showed that $t_1,\dots, t_v, s_{\epsilon+1}$ are basic in $\E$ at all points of codimension $\le m-1$. 
There are $(S_{m-1}^+)$-Serre correspondences 
\[
u: \O(-a_{\epsilon+1}) \xrightarrow{s_{\epsilon+1}} \coker \psi, \quad v:\O(-b_{\epsilon+1}) \xrightarrow{t_{\epsilon+1}} \coker \xi
\] 
such that
\[
\coker u = \frac{\E}{(t_1,\dots, t_v, s_{\epsilon+1})} \cong \frac{\E}{(s_1,\dots, s_{\epsilon+1}, t_{\epsilon+1}, t_{\epsilon+2}',\dots, t_v')} = \coker v. \qedhere
\]
\end{proof}

The converse of \Cref{Weak} is obviously true, i.e. if two $(S_m^+)$ sheaves are related in the manner above, then they belong to the same biliaison class.

\bigskip

There is a dual notion of elementary biliaisons, see for example \cite[Definition 4.7]{Buraggina99}.
A word of caution that this dual notion does \textbf{not} generalize elementary biliaisons of subvarieties. 
We say there is a \emph{dual elementary $(S_m^+)$-biliaison} from $\F$ to $\G$ if there is a pair of $(S_m^+)$-Serre correspondences $\phi: \O(-a) \to \E$ and $\psi: \O(-b) \to \E$ for some sheaf $\E$ and integers $a,b$ such that $\coker \phi \cong \F$ and $\coker \psi \cong \G$.
We remark that \Cref{Weak} holds trivially with dual elementary $(S_m^+)$-biliaisons instead, due to the fact that an $m$-reduction remains an $m$-reduction when we restrict to a summand.
Our results in \Cref{LazarsfeldRao} give stronger statements than those in \cite{Buraggina99} even for the special case of rank two reflexive sheaves in $X=\P^3_k$.

\bigskip

If we restrict to the speical case of rank one $(S_1^+)$ sheaves on $\P^n_k$, then \Cref{Weak} recovers a weak version of the structure theorem for even linkage classes of codimension two subvarieties (cf. \cite{BBM91,Nollet96}). 
In the next section, we will prove a stronger structure theorem under an additional assumption on $X$ which we now define.

\begin{definition} 
We define the following notions relative to the very ample line bundle $\O(1)$. 
\begin{enumerate}[leftmargin = *]
\item We say a sheaf $\F$ is \emph{primitive} if $\Ext^1(\F,\O(l)) = 0$ for all $l\in\mathbb{Z}$.
\item We say $X$ is \emph{primitive} if $\O_X$ is primitive.
\item Two sheaves $\F$ and $\F'$ are \emph{stably equivalent} if $\F\oplus \O(\a) \cong \F' \oplus \O(\b)$ for some $\a,\b$. 
\end{enumerate}
\end{definition}

If $H^1(\O_X) = 0$, then $X$ is primitive relative to any large enough multiple of an ample line bundle by Serre vanishing and Serre duality.
If $X$ is subcanonical, i.e. $\omega_X \cong \O(l)$ for some integer $l$, then a sheaf $\F$ is primitive iff $H^{n-1}_*(\F) = 0$ by Serre duality, where $n = \dim X$.

\bigskip

Note that $X$ is primitive iff $H^1_*(\O_X)=0$. 
Under this assumption, biliaison equivalence is also called \emph{psi-equivalence} in \cite{Hartshorne03}, and is closely related to stable equivalence.
We recall some useful facts from the same paper. 

\begin{proposition}\label{Primitive}
Suppose $X$ is primitive.
\begin{enumerate}[leftmargin = *]
\item 
If $\F$ is a descendant of $\E$ and $\G$ is a descendant of $\F$, then $\G$ is a descendant of $\E$.
\item If $\F$ and $\G$ have a common descendant, then $\F$ and $\G$ have a common ancestor.
\item The relation $\sim$ is an equivalence relation, thus coincides with biliaison equivalence.
\item If $\E$ is primitive and shares a common descendant with $\F$, then $\F$ is a descendant of $\E$. 
Thus all sheaves in the biliaison class of a primitive sheaf $\E$ are descendants of $\E$.
\item Two primitive sheaves in the same biliaison class are stably equivalent.
\item If $\E$ is primitive and $\F$ is a sheaf in the biliaison class of $\E$ that satisfy $(S_m^+)$, then $\E$ also satisfies $(S_m^+)$.
\item If $X$ is Gorenstein in codimension one $(G_1)$ and $\F$ satisfies $(S_1^+)$, then $\F$ is the descendant of a primitive sheaf. 
In particular, biliaison classes that contain an $(S_1^+)$ sheaf are in bijection with the stable equivalence classes of primitive $(S_1^+)$ sheaves. 
\end{enumerate}
\end{proposition}
\begin{proof}
\begin{enumerate}[leftmargin = *]
\item \cite[Lemma 2.4]{Hartshorne03}.
\item \cite[Lemma 2.5]{Hartshorne03}.
\item The relation $\sim$ is evidently reflexive and symmetric.
(1) and (2) show that $\sim$ is transitive. 
\item Given $0 \to \O(\a) \to \E \oplus \O(\b) \to \G \to 0$ and $0 \to \O(\a') \to \F\oplus \O(\b') \to \G \to 0$, the map $\E\oplus \O(\b) \to \G$ lifts to a map $\E\oplus \O(\b) \to \F\oplus \O(\b')$ since $\Ext^1(\E\oplus \O(\b),\O(\a')) = 0$. 
By the horseshoe lemma we have an extension 
\[
0 \to \O(\a) \to \E\oplus \O(\b)\oplus \O(\a') \to \F\oplus \O(\b')\to 0.
\]
It follows that $\F\oplus \O(\b')$ is a descendant of $\E$. 
Since $\F$ is a descendant of $\F\oplus \O(\b')$, we conclude from (1) that $\F$ is a descendant of $\E$.
\item  By (4) we have an exact sequence $0 \to \O(\a) \to \E\oplus \O(\b) \to \E' \to 0$.  
This sequence is split since $\Ext^1(\E',\O(\a)) = 0$.
We conclude that $\E\oplus \O(\b) \cong \E'\oplus \O(\a)$.
\item By (4), the sheaf $\F$ is a descendant of $\E$.
Thus $\E$ satisfies $(S_m^+)$ by the remark below \Cref{Descendant}.
\item Consider a surjection $\O(\a) \to \F$ with kernel $\K$. 
There is a surjection $\Hom_*(\K,\O_X) \to \Ext^1_*(\F,\O_X)$. 
Since $X$ satisfies $(G_1)$ and $(S_2)$, and $\K$ satisfies $(S_2)$, we see that $\K$ is reflexive.
We conclude that $\Hom_*(\K,\O_X)$ and $\Ext^1_*(\F,\O_X)$ are finitely generated modules over $H^0_*(\O_X)$. 
We may then find an extension 
\[
0 \to \O(\a) \to \E \to \F \to 0,
\]
where the map $\alpha$ in the long exact sequence
\[
\cdots \to \Hom_*(\O(\a),\O_X) \xrightarrow{\alpha} \Ext^1_*(\F,\O_X) \to \Ext^1_*(\E,\O_X) \to 0
\]
corresponds to generators of the module $\Ext^1_*(\F,\O_X)$  \cite[Prop 2.1]{Hartshorne03}. \qedhere
\end{enumerate}
\end{proof}
From the lower terms of the spectral sequence of $\Ext$ 
\[
0\to H^1_*(\E^*) \to \Ext^1_*(\E,\O_X) \to H^0_*(\sExt^1(\E,\O_X)) \to H^2_*(\E^*) \to \dots
\]
we see that \emph{extraverti} sheaves in the sense of \cite{Hartshorne03} are primitive.
The converse need not be true as extraverti sheaves are exactly primitive sheaves whose classes contain the ideal sheaf of a pure codimension two subvariety up to twist.
In our article, we are not concerned with the properties of the varieties defined by the ideal sheaves in a biliaison class, thus we resort to the more general definition of primitive sheaves.

\section{Minimal sheaves}

In this section, we assume that $X$ is primitive and Gorenstein in codimension one $(G_1)$.
We define a natural preorder among $(S_m^+)$ sheaves in the same biliaison class and prove that there is always a minimal member, generalizing the fact that there is always a minimal subvariety in an even linkage class. 
We then prove a stronger structure theorem for $(S_m^+)$ sheaves in a biliaison classe, which is an analogue of the Lazarsfeld-Rao property for even linkage classes.
Finally, we deduce a sufficient criterion for an $(S_m^+)$ sheaf to be minimal.

\begin{definition}
We say a sheaf $\F$ is \emph{very primitive} if $\F$ is primitive and does not admit a non-trivial direct summand of the form $\O(\a)$.
\end{definition}

Recall that coherent sheaves on $X$ form a Krull-Schimdt category \cite{Atiyah56}. 
In particular, any primitive sheaf $\F$ is of the form $\F'\oplus \O(\a)$ for some very primitive sheaf $\F'$ and finite integer sequence $\a$, and this decomposition is unique up to isomorphism.
It follows from \Cref{Primitive} that a very primitive sheaf is unique up to isomorphism in its biliaison class.

\bigskip

We define the following invariant for sheaves in the biliaison class of a primitive sheaf.

\begin{definition}\label{Sigma}
Let $\E$ be a very primitive sheaf, and let $\F$ be in the biliaison class of $\E$.
It follows from \Cref{Primitive} that $\F\in \D(\E)$, i.e. there is an extension of the form
\[
0 \to \O(\a) \to \E \oplus \O(\b) \to \F \to 0.
\]
We define the \emph{$\Sigma$ function} of $\F$ to be 
$
\Sigma(\F,l) := \Sigma(\b,l)-\Sigma(\a,l).
$
\end{definition}

\begin{proposition}\label{Defined}
The function $\Sigma(\F,-)$ is well-defined for any sheaf $\F$ in the biliaison class of a primitive sheaf. 
In particular, the $\Sigma$ function is well defined for any sheaf $\F$ in the biliaison class of an $(S_1^+)$ sheaf.
\end{proposition} 
\begin{proof}
We need to show that the function $\Sigma(\F,-)$ does not depend on the extension 
\[
0 \to \O(\a) \to \E \oplus \O(\b) \to \F \to 0.
\]
Suppose $0 \to \O(\a') \to \E' \oplus \O(\b') \to \F \to 0$ is another extension where $\E'$ is very primitive.  
The surjection $\E' \oplus \O(\b') \to \F$ lifts to a map $\E' \oplus \O(\b') \to \E \oplus \O(\b)$ since $\E$ is primitive and $X$ is primitive. 
We have a surjection $\E' \oplus \O(\b') \oplus \O(\a) \to \E \oplus \O(\b)$ with kernel $\O(\a')$ by the horseshoe lemma. 
Since $\Ext^1(\E\oplus \O(\b),\O(\a')) = 0$, the above surjection splits and we obtain an isormorphism
\[\E' \oplus \O(\b') \oplus \O(\a)  \cong  \E \oplus \O(\b)\oplus \O(\a').\]
Since $\E$ and $\E'$ are both very primitive, we have that $\O(\b')\oplus \O(\a) \cong \O(\b)\oplus \O(\a')$ by the uniqueness of the Krull-Schimdt decomposition.
It follows that $\Sigma(\F,l) = \Sigma(\b,l)-\Sigma(\a,l) = \Sigma(\b',l)-\Sigma(\a',l)$. 
The last statement follows from \Cref{Primitive}.
\end{proof}

Given the Hilbert function of the primitive sheaf $\E$, the data of the $\Sigma$ function of a sheaf $\F$ in the biliaison class of $\E$ is equivalent to the data of the Hilbert function of $\F$. 

For every sheaf $\F$ there is a natural surjection $\bigoplus_{l\in\mathbb{Z}} \O(-l)^{f(l)} \to \F$ given by sections of $\F$ in all degrees, where $f(l) = h^0(\F(l))$.
For any integer $a\in \mathbb{Z}$, we define $\F_{\le a}$ to be the image subsheaf of the restriction $\bigoplus_{l \le a} \O(-l)^{f(l)} \to \F$.
We say $\F_{\le a}$ is the subsheaf of $\F$ generated by sections of degree $\le a$.
Our notations are chosen to be consistent in the sense that $\Sigma(\O(\a),l) = \Sigma(\a,l) = \rank \O(\a)_{\le l}$.

\bigskip

The next proposition says we can compute the $\Sigma$ function of $\F$ from the $\Sigma$ function of any ancestor of $\F$.

\begin{proposition}\label{SigmaComp}
Suppose $\F,\G \in \D(\E)$ for some very primitive sheaf $\E$. 
If there is an extension $0 \to \O(\a) \to \G \oplus \O(\b) \to \F \to 0$, then 
\[
\Sigma(\F\oplus \O(\a),l) = \Sigma(\G\oplus \O(\b),l),\quad \forall l\in\mathbb{Z}.
\]
\end{proposition}
\begin{proof}
Given an extension $0 \to \O(\a') \to \E\oplus \O(\b') \to \F \to 0$, the map $\E\oplus \O(\b') \to \F$ lifts to a map $\E\oplus \O(\b') \to \G\oplus \O(\b)$ and we obtain an exact sequence 
\[
0 \to \O(\a') \to \E\oplus \O(\b')\oplus \O(\a) \to \G \oplus \O(\b) \to 0.
\] 
We conclude that $\Sigma(\G\oplus \O(\b),l) = \Sigma(\b',l)+\Sigma(\a,l)-\Sigma(\a',l) = \Sigma(\F\oplus \O(\a),l)$.
\end{proof}
\bigskip

Here are some simple but important observations on $\Sigma$ functions of $(S_1^+)$ sheaves.
\begin{proposition}\label{SigmaProp}
Let $m\ge 1$ and let $\F\in \D_m(\E)$ for a very primitive sheaf $\E$. 
Let $r$ be the minimal rank of all sheaves in $\D_m(\E)$, and define $e := \inf\{l \mid H^0(\E(l)) \ne 0\}$.
\begin{enumerate}[leftmargin=*]
\item If $\E\ne 0$, then $e$ is an integer.
\item $\Sigma(\F,l) = 0$ for $l \ll 0$ and $\Sigma(\F,l) \ge 0$ for all $l<e$.
\item $\Sigma(\F,l) \ge r-\rank \E$ for all $l\ge e$.
\item $\Sigma(\F,l) = \rank \F-\rank \E$ for all $l\gg 0$.  
\end{enumerate}
\end{proposition}
\begin{proof}
\begin{enumerate}[leftmargin=*]
\item Since $\E$ satisfies $(S_1^+)$, the map $\E \to \E^{**}$ is injective. 
We have $H^0(\E^{**}(l)) = 0$ for $l\ll 0$ by Serre duality and Serre vanishing. 
It follows that $H^0(\E(l)) = 0$ for $l\ll 0$ and $e$ is an integer. 
\item  Let $0 \to \O(\a) \to \E \oplus \O(\b) \to \F \to 0$ be an extension. 
It is clear that the restricted map $\phi:\O(\a)_{\le l} \to \E\oplus \O(\b)$ is an $m$-reduction. 
In fact, $\O(\a)_{\le l}$ maps into $\O(\b)_{\le l}$ since $\E \oplus \O(\b)_{>l}$ has no sections of degree $\le l$. 
It follows that $\psi:\O(\a)_{\le l} \to \O(\b)_{\le l}$ is an $m$-reduction, as its cokernel fits in an exact sequence 
\[
0\to \coker \psi \to \coker \phi \to \O(\b)_{>l} \to 0.
\]
We conclude that $\Sigma(\F,l) = \rank \O(\b)_{\le l} - \rank \O(\a)_{\le l} \ge 0$.
\item Suppose $l \ge e$. 
Similar to the above, there is an $m$-reduction $\psi:\O(\a)_{\le l} \to \E \oplus \O(\b)_{\le l}$. 
Since $\Sigma(\F,l) + \rank \E  = \rank \coker \psi \ge r$, we see that $\Sigma(\F,l)\ge r-\rank \E$.
\item This is true for any $l$ greater than the maximum of all entries of $\a$ and $\b$.  \qedhere
\end{enumerate}
\end{proof}

\bigskip

The invariant $\Sigma$ allows us to define a preorder on the biliaison class of a primitive sheaf.

\begin{definition}\label{Preorder}
If $\E$ is a very primitive sheaf and $\F,\G \in \D(\E)$, we write $\F\preceq \G$ if $\Sigma(\F,l) \le \Sigma(\G,l)$ for all $l\in\mathbb{Z}$. 
This defines a preorder on the biliaison class of $\E$.
\end{definition}

%
%
%
%
A preorder is a relation that is reflexive and transitive. 
Every preorder has an associated partial order, obtained by modding out equivalences where $\F\preceq\G$ and $\G\preceq \F$. 
The associated poset of a biliaison class with respect to the preorder $\preceq$ is exactly the poset of $\Sigma$ functions, where the partial order is given by point-wise comparison. 
The next proposition characterizes when two sheaves in a biliaison class have the same $\Sigma$ functions.

\begin{proposition}\label{Rigid}
The following are equivalent for two sheaves $\F$ and $\G$ whose biliaison classes admit primitive sheaves.
\begin{enumerate}[leftmargin = *]
\item $\F$ and $\G$ are in a rigid family,
\item $\F$ and $\G$ are in the same biliaison class, and $\F\preceq \G$ as well as $\G\preceq \F$. 
\end{enumerate}
\end{proposition}
\begin{proof}
If $\F$ and $\G$  are in the same biliaison class and have the same $\Sigma$ functions, then there are extensions 
\[
0 \to \O(\a) \to \E\oplus \O(\b) \to \F \to 0
\]
\[0 \to \O(\a) \to \E\oplus \O(\b) \to \G \to 0
\]
for a very primitive sheaf $\E$ by the proof of \Cref{Defined}.
The proof of \Cref{Replace} shows that $\F$ and $\G$ lie in a rigid family, parametrized by an open subscheme of the affine scheme $\Hom(\O(\a),\E\oplus \O(\b))$.

Conversely, if $\F$ and $\G$ lie in a rigid family, then there is a not necessarily primitive sheaf $\E$ and extensions 
\[
0 \to \O(\a) \to \E \to \F \to 0
\]
\[
0 \to \O(\a) \to \E \to \G \to 0.
\] 
It follows from \Cref{SigmaComp} that $\Sigma(\F,l) = \Sigma(\G,l) = \Sigma(\E,l)-\Sigma(\a,l)$ for all $l\in\mathbb{Z}$.
\end{proof}

\bigskip

As a direct corollary to \Cref{Main1}, we see that the associated poset of $(S_m^+)$ sheaves in a biliaison class is a meet-semilattice.
\newtheorem*{lattice2}{\Cref{Main1}2}
\begin{lattice2}\label{Lattice}
For $m\ge 1$, the $\Sigma$ functions of $(S_m^+)$ sheaves in a biliaison class form a meet-semilattice, i.e. a poset with meet. 
For $m = 1$, the $\Sigma$ functions of $(S_1^+)$ sheaves in a biliaison class form a lattice.
\end{lattice2}
\begin{proof}
Let $\F$ and $\G$ be two $(S_m^+)$ sheaves in the same biliaison class.
We find a sheaf $\E$ and $m$-reductions $\phi:\O(\a) \to \E$ and $\psi:\O(\b) \to \E$ where $\coker \phi \cong \F$ and $\coker \psi \cong \G$ by \Cref{Primitive}.
By \Cref{Join}, we find an $m$-reduction $\xi:\O(\c) \to \E$ where $\c = \a\wedge \b$. 
By \Cref{SigmaComp}, we see that $\Sigma(\coker \xi,l) = \min(\Sigma(\F,l), \Sigma(\G,l))$. 
The first conclusion follows. 
The second conclusion follows analogously from \Cref{Meet}.
\end{proof}

\bigskip

In the following, we show that the meet-semilattice of $\Sigma$ functions of $(S_m^+)$ sheaves in a biliaison class is always bounded below.

\begin{definition}
For $m\ge 1$, a \emph{minimal $(S_m^+)$ sheaf} is a sheaf that is minimal among all $(S_m^+)$ sheaves in its biliaison class with respect to the preorder $\preceq$.
\end{definition}

We make several remarks regarding this definition.

First, any two minimal $(S_m^+)$ sheaves in a biliaion class lie in a rigid family by \Cref{Rigid}.
Since we assume that $X$ is primitive in this section, all minimal $(S_m^+)$ sheaves in a biliaison class have the same intermediate cohomology modules and Hilbert functions.

Second, note that the Chern classes of minimal sheaves have smallest degrees (with respect to pairing with complementary powers of the hyperplane class $H$) among all $(S_m^+)$ sheaves in their biliaison classes. 

Third, since $\rank \F = \rank \E+\Sigma(\F,l)$ for $l\gg 0$, where $\E$ is a very primitive sheaf in the biliaison class of $\F$, 
we conclude that $\F\preceq\G$ implies $\rank \F\le \rank \G$.
In particular, minimal $(S_m^+)$ sheaves must have minimal rank among all $(S_m^+)$ sheaves in its biliaison class.

Fourth, one might ask if there is a minimal member among $(S_m^+)$ sheaves with a given rank in a biliaison class.
However, such a sheaf might not exist. 
Consider the biliaison class of the zero sheaf, one immediately sees that there is no minimal rank one bundle as we have the bundle $\O(l)$ for any $l\gg 0$.

Last but not least, in the liaison theory of subvarieties of $\P^n_k$, a variety is minimal in its even linkage class iff its ideal sheaf is minimal among all rank one $(S_1^+)$ sheaves in its biliaison class with respect to the preorder $\preceq$.
We will see that these ideal sheaves are in fact minimal among $(S_1^+)$ sheaves of \textbf{all} ranks in its biliaison class, i.e. they are minimal $(S_1^+)$ sheaves.

\bigskip

The next proposition says that in order for $\F$ to be a minimal $(S_m^+)$ sheaf, we only need to check that $\F\preceq \G$ for all $(S_m^+)$ sheaves \textbf{of minimal rank} in its biliaison class.
\begin{proposition}\label{Minimal}
If $m\ge 1$ and $\F \preceq \G$ for all $(S_m^+)$ sheaves $\G$ of minimal rank in the biliaison class of $\F$, then $\F$ is a minimal $(S_m^+)$ sheaf.
\end{proposition}
\begin{proof}
Clearly the condition implies that $\F$ has minimal rank among $(S_m^+)$ sheaves in its biliaison class. 
Suppose $\E$ is any $(S_m^+)$ sheaf in the biliaison class of $\F$, then can find a sheaf $\G$ where $\G\preceq \E$ and $\G\preceq \F$ by \Cref{Lattice}.
Since $\rank \F = \rank \G$ is minimal, by assumption we see that $\F\preceq \G$, and thus $\F\preceq \E$.
\end{proof}

 Migliore \cite{Migliore83} proved that every even linkage class of curves in $\P^3_k$ has a minimal member.
This result is then extended in \cite{BM89} to every even linkage class of pure codimension two Cohen-Macaulay subvarieties of $\P^n_k$ has a minimal member.
Nollet \cite{Nollet96} generalized this further to pure codimension $r$ subvarieties and removed the Cohen-Macaulay assumption, and described an algorithm to construct the minimal ideal sheaves given a primitive sheaf as input based on an algorithm in \cite{MDP90} for the case of space curves.
Combined with \Cref{Minimal}, the above results give us many examples of minimal sheaves.

\begin{corollary}\

\noindent There is a minimal $(S_1^+)$ sheaf in every biliaison class that admits an $(S_1^+)$ sheaf on $\P^n_k$.
\end{corollary}

We prove a generalization of the above result for $(S_m^+)$ sheaves on any projective variety $X$ satisfying our assumptions.

\begin{maintheorem}(Existence of minimal sheaves) \label{Existence}\

\noindent There is a minimal $(S_m^+)$ sheaf in every biliaison class that admits an $(S_m^+)$ sheaf if $m\ge 1$.
\end{maintheorem}
\begin{proof}
Let $\E$ be a very primitive sheaf satisfying $(S_m^+)$.
If $\E = 0$, then the zero sheaf is the minimal $(S_m^+)$ sheaf. 
If $\E\ne 0$, then the zero sheaf is not in $\D_m(\E)$. 
Since $m\ge 1$, any sheaf in $\D_m(\E)$ is torsion-free and thus has positive rank. 
Let $r$ be the minimal rank of sheaves in $\D_m(\E)$. 
Let $\F_1 \in \D_m(\E)$ be a sheaf of rank $r$. 
If $\F_1$ is not minimal, then there exists a sheaf $\G \in \D_m(\E)$ where $\F_1 \not\preceq \G$. 
By \Cref{Lattice}, there exists a sheaf $\F_2\in \D_m(\E)$ such that $\F_2\preceq \F_1$ and $\F_2 \preceq \G$. 
Since $\F\not\preceq \G$, we must have $\F_2\prec \F_1$.
Since $\rank \F_2 \le \rank \F_1$, we see that $\rank \F_2 = r$ as well. 
Suppose to the contrary that $\D_m(\E)$ has no minimal member, arguing analogously, we obtain an infinite descending chain of rank $r$ sheaves $\F_1 \succ\F_2 \succ \cdots $.
They give an infinite descending chain of $\Sigma$ functions $\Sigma(\F_1,-) > \Sigma(\F_2,-) > \cdots$.
Set $e := \inf \{l\mid H^0(\E(l))\ne 0\}$.
By \Cref{SigmaProp}, $e$ is an integer and $\Sigma(\F_i,l) = 0$ for $l \ll 0$, $\Sigma(\F_i,l) \ge 0$ for $l < e$, $\Sigma(\F_i,l) \ge r-\rank \E$ for $l\ge e$ and $\Sigma(\F_i,l) = r-\rank \E$ for $l\gg 0$.
We see that it is not possible to have an infinite descending chain of such functions $\Sigma(\F_1,-)>\Sigma(\F_2,-)>\cdots$.
The assertion follows.
\end{proof}

The existence of minimal sheaves allows us to strengthen the structure theorem for $(S_m^+)$ sheaves in a biliaison class.

\newtheorem*{strong}{\Cref{Weak}2}
\begin{strong}(Strong structure theorem)\label{LazarsfeldRao}
Suppose $X$ is primitive and $\F$ is an $(S_m^+)$ sheaf for $m\ge 1$.
There are $(S_m^+)$ sheaves $\F = \F_0,\dots, \F_l$ in the biliaison class of $\F$, such that $\F_l$ is a minimal $(S_m^+)$ sheaf, and $\F_i, \F_{i+1}$ are related in one of the following manner:
\begin{enumerate}[leftmargin=*]
\item[(a)] there is a descending elemenatary $(S_m^+)$-biliaison from $\F_i$ to $\F_{i+1}$, 
\item[(b)] $\F_i$ and $\F_{i+1}$ belong in a rigid family,
\item[(c)] there is an $m$-reduction $\phi:\O(\a)\to \F_i$ with $\coker \phi \cong \F_{i+1}$ for some $\a$.
\end{enumerate}
If $\F$ is of minimal rank among $(S_m^+)$ sheaves in its biliaison class, then we do not need (c).
\end{strong}
\begin{proof}
Let $\G$ be a minimal $(S_m^+)$ sheaf in the biliaison class of $\F$, whose existence follows from \Cref{Existence}. 
If we follow the proof of \Cref{Weak}, we see that the elementary $(S_m^+)$-biliaison involved at every step is decreasing. 
\end{proof}

In fact, when $m = 1$, we do not need deformations by rigid families in manner (b) at all. 
A proof of this can be given based on \cite[Proposition 3.6]{Hartshorne03}.

\bigskip

Although \Cref{Existence} gives us a theoretical guarantee that minimal $(S_m^+)$ sheaves exist, it does not tell us how to produce or identify them in practice. 
The next theorem solves this problem by giving a sufficient condition for a sheaf to be a minimal $(S_m^+)$ sheaf.
This generalizes the sufficient condition for a curve in $\P^3_k$ to be minimal proven in \cite{LR83}.

\begin{maintheorem}(Sufficient condition for minimal sheaves)\label{Sufficient}
Let $m\ge 1$, and let $\F$ be an $(S_m^+)$ of minimal rank in its biliaison class. 
If $\F$ admits an extension 
\[
0 \to \O(\a) \to \E \to \F \to 0
\]
where $\E$ is primitive and $H^0(\F(l)) = 0$ for all $l< \max(\a)$, then $\F$ is a minimal $(S_m^+)$ sheaf.
\end{maintheorem}
\begin{proof}
If $\G$ is another $(S_m^+)$ sheaf in the biliaison class of $\F$, then it admits an extension of the form
\[
0 \to \O(\c) \to \E \oplus \O(\d) \to \G \to 0
\]
for some  $\c$ and $\d$ since $\E$ is primitive.
By \Cref{SigmaComp}, we need to show that 
\[
\rank (\O(\a) \oplus \O(\d))_{\le l} \ge \rank \O(\c)_{\le l},\quad \forall l\in\mathbb{Z}.
\]
We separate into two cases, where $l\ge \max(\a)$ and $l< \max(\a)$. 

\textbf{Case} $l \ge \max(\a)$: 
We have an exact sequence $0 \to \O(\c)_{\le l} \to \E \oplus \O(\d) \to \G' \to 0$, where $\G'$ is an extension of $\G$ with $\O(\c)_{>l} := \O(\c)/\O(\c)_{\le l}$.
In particular, the sheaf $\G'$ satisfies $(S_m^+)$. 
Since any map $\O(\c)_{\le l} \to \O(\d)_{>l} := \O(\d)/\O(\d)_{\le l}$ is zero, the injection $\O(\c)_{\le l} \to \E \oplus \O(\d)$ lands inside $\E \oplus \O(\d)_{\le l}$.
We obtain an exact sequence $0 \to \O(\c)_{\le l} \to \E \oplus \O(\d)_{\le l} \to \G'' \to 0$, where $\G''$ sits in an exact sequence $0 \to \G'' \to \G' \to \O(\d)_{>l} \to 0$.
By the depth lemma, the sheaf $\G''$ also satisfies $(S_m^+)$.
Now $\rank \O(\a)_{\le l} = \rank \O(\a) = \rank \E -\rank \F$ by the assumption on $a$.
Since $\F$ has minimal rank among sheaves in $\D_m(\E)$, we conclude that
\[
\rank \E + \rank \O(\d)_{\le l} - \rank \O(\c)_{\le l}= \rank \G'' \ge \rank \F = \rank \E -\rank \O(\a)_{\le l}.
\] 
It follows that $\rank (\O(\a)\oplus \O(\d))_{\le l} \ge \rank \O(\c)_{\le l}$.

\textbf{Case} $l < \max(\a)$: 
The cokernel $\F'$ of $\O(\a)_{\le l} \to \E$ is an extension of $\F$ by $\O(\a)_{>l}$. 
Since $H^0(\F(n)) = 0$ for all $n\le l$, the same is true for $\F'$.
We have the exact sequence
\[
0 \to \O(\a)_{\le l} \oplus \O(\d)_{\le l} \to \E \oplus \O(\d) \to \F' \oplus \O(\d)_{>l} \to 0.
\]
The composition $\O(\c)_{\le l} \to \E \oplus \O(\d) \to  \F' \oplus \O(\d)_{>l}$ is zero since $\F'$ has no sections in degree $\le l$.
It follows that the injection $\O(\c)_{\le l} \to \E$ factors through $\O(\a)_{\le l} \oplus \O(\d)_{\le l}$, and we conclude that $\rank (\O(\a) \oplus \O(\d))_{\le l} \ge 
 \rank \O(\c)_{\le l}$.
\end{proof}

The next theorem is a necessary condition for an $(S_m^+)$ sheaf of mimimal rank in its biliaison class to be a minimal $(S_m^+)$ sheaf.
\begin{maintheorem}[Necessary condition for minimal sheaves]\label{Necessary} \

\noindent Let $0\to \O(\a) \to \E  \to \F \to 0$ be an extension, where $\E$ is primitive of rank $\ge m$ and $\F$ is a minimal $(S_m^+)$ sheaf. 
If $\O(\c) \to \E$ is any surjection, then $\O(\c')\preceq \O(\a)$, where $\c'$ consists of the largest $u:=\rank \E-m $ entries of $\c$.
\end{maintheorem}
\begin{proof}
By \cite[Theorem 2.5]{Basic}, there is always an $m$-reduction $\phi:\O(\c')\to \E$ whose cokernel has rank $m$. 
It follows that $\O(\c')\preceq\O(\a)$ since $\F$ is a minimal $(S_m^+)$ sheaf.
\end{proof}
We remark that the necessary condition in \Cref{Necessary} is not tight in general.
The following is an example on how one could use \Cref{Necessary} in practice.
\begin{example}
Suppose $\F$ is a sheaf of minimal rank $r$ in $\D_m(\E)$, where $\E$ is primitive of rank $\ge m$.
Let there be an extension of the form 
\[
0 \to \O(-2)^v \to \E \to \F \to 0.
\]
If $\E$ is generated in degree 1, then there is a surjection $\O(-1)^N \to \E$ for some large $N$. 
Since there is an $m$-reduction of the form $\phi: \O(-1)^{u} \to \E$, where $u =\rank \E -m \le \rank \E -r = v$, \Cref{Necessary} says that $\F$ cannot be a minimal $(S_m^+)$ sheaf since there must be an $m$-reduction of $\E$ with shape $(\underbrace{1,\dots,1}_{u}) \wedge(\underbrace{2,\dots, 2}_{v}) = (\underbrace{1,\dots,1}_{u},\underbrace{2,\dots, 2}_{v-u})$.
\end{example}

%
%
%

\section{Applications}
In this section, we study some applications of biliaison to vector bundles on $X = \P^n_k$.

We say a bundle is minimal if it is a minimal $(S_n^+)$ sheaf where $n = \dim X$.

\begin{example}
If $X = \P^2_k$, then every bundle $\F$ in is in the biliaison class of the zero sheaf $\F$ admits a resolution of the form $0 \to \O(\a) \to \O(\b) \to \F \to 0$. 
More generally, on $X = \P^n_k$, a bundle $\F$ is in the biliaison class of the zero sheaf iff $H^i_*(\F) = 0$ for $1\le i\le n-2$.
Since the $\Sigma$ function of a bundle determines and is determined by its Hilbert function $H$, we see that the Hilbert functions $H$ of such bundles form a meet-semilattice in a suitable partial order, and any two bundles with the same Hilbert function $H$ lie in a rational family that preserves intermediate cohomology modules. 
All possible Hilbert functions $H$ of such bundles on $\P^n_k$ were classified in \cite{Bundles}, and the moduli $\M_{0,H}$ of bundles with Hilbert function $H$ in the biliaison class of the zero sheaf was described.
\end{example}

The next theorem is due to Buraggina \cite{Buraggina99}, based on computations in \cite{MDP92}.
We provide a conceptual proof of the same result.

\begin{theorem*}[{Buraggina \cite{Buraggina99}}]
A rank two bundle $\E$ on $\P^3_k$ is a minimal $(S_2^+)$ (i.e. reflexive) sheaf if and only if it is indecomposable.
\end{theorem*}
\begin{proof}
If $\E$ is decomposable, then it is the direct sum of two line bundles. 
In this case, the minimal reflexive sheaf in the class of $\E$ is the zero sheaf.
Suppose $\E$ is indecomposable, then the zero sheaf is not in the biliaison class of $\E$ since $H^1_*(\E)\ne 0$.
First we show that $\E$ is a minimal bundle.
By \Cref{Minimal}, we only need to show that $\E \preceq \E'$ for any other rank two bundle $\E'$ in its biliaison class.

Let $M := H^1_*(\E)$ and $c_1 := c_1(\E)$.
The Horrocks' technique of eliminating homology shows that there are universal extensions killing $H^1_*(\E)$ and $H^2_*(\E)$
\[
0 \to \E \to \F \to \O(\a) \to 0
\]
\[
0\to \O(-\a+c_1) \to \G \to \E \to 0,
\]
which fit into the display 
\[
\begin{tikzcd}[column sep=small, row sep = small]
 & & 0 \arrow[d] & 0 \arrow[d]\\
 0 \arrow[r]& \O(-\a+c_1)\arrow[r] \arrow[d,equal]& \G \arrow[r]\arrow[d,"d"]& \E \arrow[d]\arrow[r]& 0\\
 0 \arrow[r]& \O(-\a+c_1)  \arrow[r]& \H \arrow[d] \arrow[r]& \F\arrow[d] \arrow[r]& 0\\
 & & \O(\a) \arrow[r,equal]\arrow[d]& \O(\a) \arrow[d]&\\
  & & 0 & 0 &
\end{tikzcd}
\]
of a monad $\O(-\a+c_1) \to \H \to \O(\a)$ \cite{BH78}. 
By Horrocks' criterion of splitting, we see that $\H \cong \O(\b)$ for some $\b$ as it has no $H^1_*$ nor $H^2_*$. 
If we chose $H^0_*(\O(\a))\to M$ to be a minimal system of generators, then a result in \cite{HR91} shows that the map $d:\G \to \H$ would not split off summands.
It follows that $0 \to H^0_*(\G) \to H^0_*(\H) \to H^0_*(\O(\a) \to M$ are the first steps of a minimal free resolution of $M$, and $\G \cong \tilde{\Omega^2M}$, where $\Omega^2M$ denotes the second minimal syzygy of $M$. 
The same statements apply to $\E'$, and $H^1_*(\E') \cong H^1_*(\E) =M$.
Since $\G$ is primitive, we conclude from \Cref{SigmaComp} that $\Sigma(\E,l) = \Sigma(\E',l)$.

If $\F$ is a reflexive sheaf in the biliaison class of $\E$ such that $\F\preceq\E$, then $\F$ has rank at most two. 
Since $\F$ is neither zero or $\O(l)$, it must have rank exactly two.
Since $\E$ is obtained using finitely many ascending elementary $(S_2^+)$-biliaison and rigid deformations by \Cref{LazarsfeldRao}, we see that $c_3\F \le c_3\E$. 
It follows from \cite[Prop. 2.6]{Hartshorne80} that $\F$ is in fact a bundle, and thus $\E\preceq \F$ as well by the above.
\end{proof}

\bigskip

Let us summarize the situation on $\P^3_k$. 
The finite length module $M = H^1_*(\E)$ uniquely determines the stable equivalence class of a primitive bundle $\E$  \cite{Horrocks64}, and therefore uniquely determines the biliaison class of a bundle $\E$ on $\P^3_k$.
There are three possibilities.
\begin{enumerate}[leftmargin=*]
\item The minimal bundles of a biliaison class have rank two if and only if $M$ satisfies the condition in \cite{Decker90}.
\item The minimal bundle of a biliaison class is the zero sheaf if and only if $M = 0$. 
\item The minimal bundles of all other biliaison classes have rank three. 
\end{enumerate}

\bigskip

Perhaps not surprisingly, we show that the Horrocks-Mumford bundle is minimal.
\begin{theorem}
The Horrocks-Mumford bundle $\F$ on $\P^4_\mathbb{C}$ is minimal. 
\end{theorem}
\begin{proof}
Let $H \subseteq \op{SL}(5,\mathbb{C})$ be the Heisenberg group.
Let $V = \op{Map}(\mathbb{Z}/5,\mathbb{C})$ and let $V_1,\dots, V_4$ be the four irreducible representations of $H$ arising from $V$ as in \cite{HM73}.
Let $W = \Hom_H(V_1,\wedge^2 V)$.
The Horrocks-Mumford bundle $\F$ is the homology of the monad
\[
\O(2)\otimes V_1 \xrightarrow{p} \wedge^2 \mathcal{T} \otimes W\xrightarrow{q} \O(3)\otimes V_3.
\]
We show that $\ker q$ is primitive. 
By the short exact sequence 
\[
0 \to \ker q \to  \wedge^2 \mathcal{T} \otimes W\to \O(3)\otimes V_3 \to 0,
\]
it suffices to show that $\wedge^2 \mathcal{T}$ is primitive. 
Consider the Koszul complex
\[
0\to \O \to \O(1)\otimes V \to \O(2)\otimes \wedge^2 V \to \O(3)\otimes \wedge^3 V \xrightarrow{d} \O(4)\otimes \wedge^4 V \to \O(5)\otimes \wedge^5 V \to 0,
\]
where $\ker d \cong \wedge^2 \mathcal{T}$.
We see that $\Ext^2(\im d, \O(l)) = 0$ for all $l$ by the short exact sequence 
\[
0 \to \im d \to \O(4)\otimes \wedge^4 V \to \O(5)\otimes \wedge^5 V \to 0,
\]
and thus $\Ext^1(\wedge^2\mathcal{T},\O(l)) = 0$ for all $l$ by the short exact sequence
\[
0 \to \wedge^2 \mathcal{T} \to \O(3)\otimes \wedge^3 V  \to \im d \to 0.
\]
Finally, we have $H^0(\F(l)) = 0$ for $l< 0$ \cite[\S 4]{HM73}. 
Since the maximum degree of $\O(2)\otimes V_1$ is $-2$, the conclusion follows from \Cref{Sufficient} applied to the extension
\[
0 \to \O(2)\otimes V_1 \to \ker q \to \F \to 0. \qedhere
\]
\end{proof}

Since rank one reflexive sheaves on $\P^n_k$ are just line bundles $\O(l)$, and $\F$ is not in the biliaison class of the zero sheaf, the above proof shows that $\F$ is in fact a minimal $(S_2^+)$ (i.e. reflexive) sheaf in its biliaison class.
Note that $H^0(\ker q(-2)) \cong V_1$, and therefore all minimal bundles in the biliaison class of $\F$ are equivalent under the action of $\op{PGL}(5,\mathbb{C})$.
In particular, all reflexive sheaves in the biliaison class of the Horrocks-Mumford bundle are constructed from it using finitely many steps of ascending elementary biliaisons, rigid deformations and extensions by line bundles.

\bigskip
The stable equivalence of primitive bundles $\E$ on $\P^4_k$ are completely classified by the $S$-modules $M := H^1_*(\E)$, $N := H^2_*(\E)$ and an element $\xi\in \Ext^2_S(M,N)$ up to isomorphisms of this data \cite{Horrocks64}, where $S := k[x_0,\dots, x_4]$. 
In particular, the invariant $(M,N,\xi)$ characterizes biliaison classes of bundles.

\begin{question}
Can one characterize, in terms of the invariant $(M,N,\xi)$, when minimal bundles of a biliaison class on $\P^4_k$ have rank $r = 2,3$?
\end{question}

This question is conceivably difficult to answer when $r = 2$ due to a lack of examples of rank two bundles in $\P^4_k$.

\nocite{*}
\bibliographystyle{unsrt}
\bibliography{Biliaison_of_sheaves}

\end{document}